\documentclass[11pt,reqno]{amsart}
\usepackage{amsmath,amssymb,extarrows}
\usepackage{url}
\usepackage{tikz,enumerate}
\usepackage{diagbox}
\usepackage{appendix}
\usepackage{epic}
\usepackage{cite}

\usepackage{hyperref}
\usepackage{array}

\usepackage{booktabs}

\allowdisplaybreaks[4]

\newtheorem{theorem}{Theorem}
\newtheorem{corollary}[theorem]{Corollary}
\newtheorem{conj}[theorem]{Conjecture}

\theoremstyle{definition}

\theoremstyle{remark}
\newtheorem{rem}{Remark}
\numberwithin{equation}{section}
\numberwithin{defn}{section}
\begin{document}
\title[Truncated series with Nonnegative coefficients]
 {Truncated Series with Nonnegative Coefficients from the Jacobi Triple Product}

\author{Liuquan Wang}
\address{School of Mathematics and Statistics, Wuhan University, Wuhan 430072, Hubei, People's Republic of China}
\email{wanglq@whu.edu.cn;mathlqwang@163.com}

%\date{May 12, 2016}
\dedicatory{Dedicated to the memory of Srinivasa Ramanujan}
%\thanks{}



\begin{abstract}
Andrews and  Merca investigated a truncated version of Euler's pentagonal number theorem and showed that the coefficients of the truncated series are nonnegative.  They also considered the truncated series arising from Jacobi's triple product identity, and they that its coefficients are nonnegative. This conjecture was posed by Guo and Zeng independently and confirmed by  Mao and  Yee using different approaches. In this paper, we provide a new combinatorial proof of their nonnegativity result related to Euler's pentagonal number theorem. Meanwhile, we find an analogous result for a truncated series arising from Jacobi's triple product identity in a different manner.
\end{abstract}

\maketitle
\section{Introduction}	
A partition of a positive integer $n$ is any non-increasing sequence of positive integers which add up to $n$. For example, $\lambda=(5,3,1,1)$ is a partition of $10$. We denote the number of partitions of $n$ by $p(n)$ and agree that $p(0)=1$ for convention. The generating function of $p(n)$ is
\begin{equation}\label{p(n)gen}
\sum\limits_{n=0}^{\infty}p(n)q^n=\frac{1}{(q;q)_{\infty}}.
\end{equation}
Here and for the rest of this paper, we assume $|q|<1$ and use the notation
\begin{align}
(x_1,x_2,\cdots, x_k;q)_{m}=\prod\limits_{n=0}^{m-1}(1-x_1q^n)(1-x_2q^n)\cdots (1-x_kq^n),\quad m\in \mathbb{N}\cup \{\infty\}.
\end{align}

One of the most famous $q$-series identities is Euler's pentagonal number theorem:
\begin{equation}\label{Euler}
(q;q)_{\infty}=\sum\limits_{j=0}^{\infty}(-1)^jq^{j(3j-1)/2}(1-q^{2j+1}).
\end{equation}
One can easily derive the following recurrence relation for $p(n)$ from \eqref{p(n)gen} and \eqref{Euler} \cite[Corollary 1.3.6]{Berndt}
\begin{equation}\label{Euler-rec}
\sum\limits_{j=0}^{\infty}(-1)^j\Big(p(n-j(3j+1)/2)-p(n-j(3j+5)/2-1) \Big) =0,
\end{equation}
where we assume $p(m)=0$ for $m<0$.

In 2012, Andrews and Merca \cite{Andrews-Merca2012} studied the truncated series arising from \eqref{Euler}. They established the following interesting identity:
\begin{align}\label{AM-id}
\frac{1}{(q;q)_\infty}\sum_{j=0}^{k-1}(-1)^jq^{j(3j+1)/2}(1-q^{2j+1})=1+(-1)^{k-1}\sum_{n=k}^\infty \frac{q^{\binom{k}{2}+(k+1)n}}{(q;q)_k}\left[\begin{matrix}
n-1 \\ k-1
\end{matrix}\right].
\end{align}
Here the $q$-binomial coefficient is defined as
\begin{align}
\left[\begin{matrix}
n \\k
\end{matrix}\right]=\left[\begin{matrix}
n \\k
\end{matrix}\right]_q:=\left\{\begin{array}{ll}
\frac{(q;q)_n}{(q;q)_k(q;q)_{n-k}}, &\text{if $0\leq k \leq n$}, \\
0, & \text{otherwise}.
\end{array}\right.
\end{align}
As a consequence, they proved that for $k\ge 1$,
\begin{align}\label{Andrews-Merca-id}
(-1)^{k-1}\sum\limits_{j=0}^{k-1}(-1)^j\Big(p(n-j(3j+1)/2)-p(n-j(3j+5)/2-1) \Big)=M_k(n),
\end{align}
where $M_k(n)$ is the number of partitions of $n$ in which $k$ is the least integer that is not a part and there are more parts $>k$ than there are $<k$. It follows that for $n>0$ and $k\geq 1$ (see \cite[Corollary 1.3]{Andrews-Merca2012}),
\begin{equation}\label{Euler-ineq}
(-1)^{k-1}\sum\limits_{j=0}^{k-1}(-1)^j\Big(p(n-j(3j+1)/2)-p(n-j(3j+5)/2-1) \Big) \ge 0.
\end{equation}
Andrews and Merca \cite[Question (1)]{Andrews-Merca2012} asked whether one could give a combinatorial proof of \eqref{Andrews-Merca-id} hopefully characterizing the partitions remaining after a sieving process.

As a generalization of \eqref{Euler}, Jacobi's triple product identity states that for $z\ne 0$,
\begin{equation}\label{JTP}
\sum\limits_{j=-\infty}^{\infty}(-z)^jq^{j(j-1)/2}=(z;q)_{\infty}(q/z;q)_{\infty}(q;q)_{\infty}.
\end{equation}
By considering truncated series arising from Jacobi's triple product identity, Andrews and Merca \cite[Question (2)]{Andrews-Merca2012} and Guo and Zeng \cite[Conjecture 6.1]{Guo} posed the following conjecture:
\begin{conj}\label{AMconj}
For positive integers $k,R,S$ with $k\ge 1$ and $1\le S<R/2$, the coefficient of $q^n$ with $n\ge 1$ in
\begin{equation}\label{Andrewsconj}
(-1)^{k-1}\frac{1}{(q^S,q^{R-S},q^{R};q^{R})_{\infty}}\sum\limits_{j=0}^{k-1}(-1)^jq^{Rj(j+1)/2-Sj}(1-q^{(2j+1)S})
\end{equation}
is nonnegative.
\end{conj}
Clearly \eqref{Euler-ineq} is the special case $(R,S)=(3,1)$ of this conjecture. Guo and Zeng \cite{Guo} provided more evidences to this conjecture, For example, they proved this conjecture for the case $(R,S)=(4,1)$ and a weaker inequality for the case $(R,S)=(2,1)$. Note that for $(R,S)=(4,1)$ and $(2,1)$, the partition functions generated by the infinite products in \eqref{Andrewsconj} are $\mathrm{pod}(n)$ and $\overline{p}(n)$, respectively. Here $\mathrm{pod}(n)$ counts the number of partitions of $n$ with odd parts distinct, and $\overline{p}(n)$ enumerates the number of overpartitions of $n$. Specifically speaking, Guo and Zeng proved the following inequalities in analogy with \eqref{Euler-ineq}: for $n,k\geq 1$,
\begin{align}
&(-1)^{k-1}\sum_{j=0}^{k-1}(-1)^j\Big(\mathrm{pod}(n-j(2j+1))-\mathrm{pod}(n-(j+1)(2j+1)) \Big)\geq 0, \label{pod-ineq} \\
&(-1)^{k}\Big(\overline{p}(n)+2\sum_{j=1}^{k}(-1)^j\overline{p}(n-j^2)\Big)\geq 0 . \label{overpartition-ineq}
\end{align}
They proved the above inequalities by establishing identities analogous to \eqref{AM-id}. Andrews and Merca \cite{Andrews-Merca2018} pointed out that \eqref{AM-id} and analogous identities for $\overline{p}(n)$ and $\mathrm{pod}(n)$ in the work of Guo and Zeng \cite{Guo}  are essentially corollaries of the Rogers-Fine identity.

In 2015, two different proofs of Conjecture \ref{AMconj} were given by Mao \cite{Mao} and Yee \cite{Yee} independently. Mao's proof is algebraic while Yee's proof is combinatorial. Mao \cite{Mao} also proved that Conjecture \ref{AMconj} is true for all $k\ge 1$ and $1\le S<R$.

The inequality \eqref{Euler-ineq} is not an isolated phenomenon. Besides the general result given in Conjecture \ref{AMconj}, Guo and Zeng \cite{Guo} also provided another example. They considered the partition function $t(n)$ defined by
\[\sum\limits_{n=0}^{\infty}t(n)q^n=\frac{1}{(q;q)_{\infty}^3}.\]
Note that $t(n)$ enumerates the number of partition triples (or equivalently, 3-colored partitions) of $n$. From Jacobi's identity
\[(q;q)_{\infty}^3=\sum\limits_{j=0}^{\infty}(-1)^j(2j+1)q^{j(j+1)/2},\]
One easily finds the recurrence formula
\begin{equation}\label{t(n)rec}
\sum\limits_{j=0}^{\infty}(-1)^j(2j+1)t(n-j(j+1)/2)=0,
\end{equation}
where $t(m)=0$ for negative $m$. Guo and Zeng \cite[Conjecture 6.4]{Guo} conjectured that for $n,k\ge 1$, there holds
\begin{align}\label{GZ-conj-1}
(-1)^k\sum\limits_{j=0}^{k}(-1)^j(2j+1)t(n-j(j+1)/2)\ge 0.
\end{align}

For convenience, we introduce a notation. For any Laurent series
\begin{align}
f(q)=\sum_{n\in \mathbb{Z}}a_nq^n \quad \text{and} \quad g(q)=\sum_{n\in \mathbb{Z}}b_nq^n
\end{align}
with real coefficients, we say that $f(q)\succeq g(q)$ when $a_n\geq b_n$ holds for all $n\in \mathbb{Z}$. With this notation, \eqref{GZ-conj-1} can also be stated as
\begin{align}\label{GZ-conj-2}
\frac{(-1)^k}{(q;q)_\infty^3}\sum_{j=0}^k(-1)^j(2j+1)q^{j(j+1)/2} \succeq 0.
\end{align}
This conjectured was confirmed by Mao \cite{Mao} using $q$-series manipulations and then proved by He, Ji and Zang \cite{HJZ} using combinatorial arguments. Recently, Wang and Yee \cite{WangYee} proved that for $1\leq S\leq R/2$ and $m\geq 1$,
\begin{align}\label{Wang-Yee-id}
&\frac{1}{(q^R,q^S,q^{R-S};q^R)_\infty} \sum_{n=0}^{m-1}(-1)^nq^{\binom{n+1}{2}R-nS}(1-q^{(2n+1)S}) \nonumber \\
=&1+(-1)^{m-1}q^{\binom{m}{2}R}\sum_{n=m}^\infty \sum_{\begin{smallmatrix}
i+j+h+k=n \\ i,j,h,k\geq 0
\end{smallmatrix}} \frac{q^{(mj+hk)R+(h-k)S+nR}}{(q^R;q^R)_i(q^R;q^R)_j(q^R;q^R)_h(q^R;q^R)_k}\left[\begin{matrix}
n-1 \\ m-1
\end{matrix}\right]_{q^{R}}.
\end{align}
This gives a reminiscent of \eqref{AM-id} and reproves Conjecture \ref{AMconj}. They also gave explicit series form of the truncated series in \eqref{GZ-conj-2}, which reconfirms \eqref{GZ-conj-2}. We remark here that there are other works related to this topic. For example, Chan, Ho and Mao \cite{ChanHoMao} investigated the truncated series arising from the quintuple product identity.

Inspired by their works, there are two goals of this paper. First, we want to give a new proof for \eqref{Euler-ineq} by giving a new combinatorial interpretation to its left side. To describe our result, we need few more notations. Let $P(m)$ be the set of all partitions of $m$ and agree that $P(m)=\emptyset$ for $m\leq 0$. For a finite set $A$, let $|A|$ denote the number of elements in $A$. Recall Dyson's rank of a partition $\lambda$, denoted as $\mathrm{rank}(\lambda)$, is defined as the largest part minus the number of parts in $\lambda$. For any  nonnegative integer $n$ and integer $j$, we denote
\begin{align}
A_{j}^{(1)}(n):=\{\lambda \in P(n-j(3j+1)/2): \mathrm{rank}(\lambda)\leq 3j\}, \label{Aj1-defn}\\
A_{j}^{(2)}(n):=\{\lambda \in P(n-j(3j+1)/2): \mathrm{rank}(\lambda)> 3j\}. \label{Aj2-defn}
\end{align}
We proved the following theorem, which reproves \eqref{Euler-ineq}.
\begin{theorem}\label{thm-Euler-ineq}
For $n>0$, $k\geq 1$,
\begin{align}\label{Ak}
&(-1)^{k-1}\sum\limits_{j=0}^{k-1}(-1)^j\Big(p(n-j(3j+1)/2)-p(n-j(3j+5)/2-1)\Big) \nonumber \\
=& |A_{k-1}^{(2)}(n)|-|A_{-k}^{(1)}(n)|.
\end{align}
Furthermore, we have $|A_{k}^{(1)}(n)|=|A_{k-1}^{(2)}(n)|\geq |A_{-k}^{(1)}(n)|$.
\end{theorem}
Recall that Kolitsch and Burnette \cite{Kolitsch} interpreted the partial sum on the left side of \eqref{Euler-ineq} using partition pairs. Compared with their interpretation, the advantage of \eqref{Ak} is that we do not need to use partition pairs. Instead, we only need to consider the partitions enumerated by $p(n-j(3j+1)/2)$, which is more natural and closer to the sieving process that Andrews and Merca \cite{Andrews-Merca2012} asked for.

The second goal of this paper is to provide more instances of nonnegative series arising from truncated sums.  Our question of motivation is: if we start with a new recurrence relation for the partition function $p(n)$ and study the truncated sums, can we find similar nonnegativity result?

A natural candidate of the recurrence relation can be found in the following way. Taking logarithmic differentiation with respect to $z$ on both sides of \eqref{JTP}, we obtain
\begin{align}\label{JTP-diff}
\frac{\sum_{j=-\infty}^\infty (-1)^jjz^jq^{j(j-1)/2}}{(z,q/z,q;q)_\infty}=\sum_{n=0}^\infty \left( \frac{q^{n+1}/z}{1-q^{n+1}/z}-\frac{zq^n}{1-zq^n} \right).
\end{align}
Replacing $q$ by $q^R$ and setting $z=q^{R-S}$, we obtain
\begin{align}\label{diff-rec}
\frac{\sum_{j=-\infty}^\infty (-1)^jjq^{Rj(j+1)/2-Sj}}{(q^S,q^{R-S},q^R;q^R)_\infty}=\sum_{n=0}^\infty \left(\frac{q^{nR+S}}{1-q^{nR+S}}-\frac{q^{nR+R-S}}{1-q^{nR+R-S}}  \right).
\end{align}
If we set $(R,S)=(3,1)$, then we obtain the following recurrence relation:
\begin{equation}\label{pn-new-rec}
\sum\limits_{j(3j+1)/2\le n, j\in \mathbb{Z}}(-1)^{j}jp\Big(n-j(3j+1)/2 \Big)=d_{1,3}(n)-d_{2,3}(n),
\end{equation}
where $d_{r,3}(n)$ is the number of divisors of $n$ which are congruent to $r$ modulo $3$. Considering partial sums in \eqref{pn-new-rec}, by computation, it appears that
\begin{align*}
&p(n-1)\geq d_{1,3}(n)-d_{2,3}(n), \\
&p(n-1)-p(n-2)-2p(n-5)\leq d_{1,3}(n)-d_{2,3}(n), \\
&p(n-1)-p(n-2)-2p(n-5)+2p(n-7)+3p(n-12)\geq d_{1,3}(n)-d_{2,3}(n), \\
%&p(n-1)-p(n-2)-2p(n-5)+2p(n-7)+3p(n-12)-3p(n-15)-4p(n-22)-\left(d_{1,3}(n)-d_{2,3}(n)\right)\leq 0, \\
&\cdots
\end{align*}
This observation motivates us to obtain some inequalities similar to \eqref{Euler-ineq}, \eqref{pod-ineq}, \eqref{overpartition-ineq} and \eqref{GZ-conj-1}. In fact, if we study truncated sums in the series in \eqref{diff-rec}, we can obtain the following general result in analogy with Conjecture \ref{AMconj}.
\begin{theorem}\label{thm-main}
For positive integers $k,R,S$ with $k\ge 1$ and $1\le S<R$, the coefficient of $q^n$ with $n\ge 1$ in
\[(-1)^{k-1}\Bigg(\frac{\sum\limits_{j=-k}^{k-1}(-1)^jjq^{Rj(j+1)/2-Sj}}{(q^S,q^{R-S},q^R;q^R)_{\infty}}-\sum\limits_{n=0}^{\infty}\Big(\frac{q^{nR+S}}{1-q^{nR+S}}-\frac{q^{nR+R-S}}{1-q^{nR+R-S}} \Big) \Bigg) \]
is nonnegative.
\end{theorem}
As a special case, setting $(R,S)=(3,1)$, we confirm our aforementioned observation.
\begin{corollary}\label{cor-pn}
Let $k\ge 1$ be a positive integer.
\begin{enumerate}
\item[$(1)$] If $k$ is odd, we have
\[\sum\limits_{j=-k}^{k-1}(-1)^jjp(n-j(3j+1)/2)\ge d_{1,3}(n)-d_{2,3}(n).\]
\item[$(2)$] If $k$ is even, we have
\[\sum\limits_{j=-k}^{k-1}(-1)^jjp(n-j(3j+1)/2)\le d_{1,3}(n)-d_{2,3}(n).\]
\end{enumerate}
\end{corollary}
%It would be interesting if one can find a combinatorial proof for Theorem \ref{thm2}.

The paper is organized as follows. In Section 2, we first sketch the combinatorial proof of the recurrence relation \eqref{Euler-rec} given by Bressoud and  Zeilberger \cite{Bressoud}, and then we prove Theorem \ref{thm-Euler-ineq}. In Section 3, we follow the techniques used by Mao \cite{Mao} to prove Theorem \ref{thm-main} via $q$-series manipulations. We will pose two open problems in the end of this paper.

\section{Proof of Theorem \ref{thm-Euler-ineq}}
Bressoud and Zeilberger \cite{Bressoud} discovered a beautiful bijective proof of Euler's partition recurrence \eqref{Euler-rec}. We sketch their proof here and use it to give a short explanation to \eqref{Euler-ineq}.

Let $b(j)=j(3j+1)/2$. Then \eqref{Euler-rec} is equivalent to
\begin{align}\label{equivalent}
\sum\limits_{\text{\rm{$j$ even}}}p(n-b(j))=\sum\limits_{\text{\rm{$j$ odd}}}p(n-b(j)).
\end{align}
 To prove \eqref{equivalent}, it suffices to establish a bijection between the two sets of partitions enumerated on both sides. Following Bressoud and Zeilberger, for any partition $\lambda=(\lambda_1,\cdots,\lambda_{t})\in P(n-b(j))$ with $\lambda_1\geq \lambda_2\geq \cdots \geq \lambda_t$, we define
\begin{align*}
\phi(\lambda):=\left\{\begin{array}{ll}
(t+3j-1,\lambda_{1}-1,\lambda_{2}-1,\cdots, \lambda_{t}-1) & \text{\rm{Case 1: if $t+3j\ge \lambda_{1}$}} \\
(\lambda_{2}+1,\lambda_{3}+1,\cdots , \lambda_{t}+1,1,\cdots,1) & \text{\rm{Case 2: if $t+3j<\lambda_{1}$}}
\end{array}\right.
\end{align*}
where there are $\lambda_{1}-(t+3j)-1$ copies of 1s in Case 2. Note that in Case 1, $\phi$ maps the elements in $P(n-b(j))$ to $P(n-b(j-1))$. In Case 2, $\phi$ maps the elements in $P(n-b(j))$ to $P(n-b(j+1))$.
It is not difficult to show that $\phi^2=1$ and thus $\phi$ is a bijection between $\cup_{\text{\rm{$j$ even}}}P(n-b(j))$ and $\cup_{\text{\rm{$j$ odd}}}P(n-b(j))$. This proves \eqref{equivalent} and hence Euler's pentagonal number theorem.

We are now ready to give a proof of \eqref{Euler-ineq} in a combinatorial way.
\begin{proof}[Proof of Theorem \ref{thm-Euler-ineq}]
Recall the definitions of $A_{j}^{(1)}(n)$ and $A_j^{(2)}(n)$ in \eqref{Aj1-defn} and \eqref{Aj2-defn}. It follows that $P(n-b(j))=A_{j}^{(1)}(n)\cup A_{j}^{(2)}(n)$. From the definition of $\phi$, we see that $\phi$ maps a partition in  $A_{j}^{(1)}(n)$ to a partition in $A_{j-1}^{(2)}(n)$, and $\phi$ maps a partition in $A_{j-1}^{(2)}(n)$ to a partition in $A_{j}^{(1)}(n)$. Since $\phi$ is a bijection, we deduce that
\begin{align}
|A_{j}^{(1)}(n)|=|A_{j-1}^{(2)}(n)|.
\end{align}
 We have
\begin{align}
&(-1)^{k-1}\sum\limits_{j=0}^{k-1}(-1)^j\Big(p(n-j(3j+1)/2)-p(n-j(3j+5)/2-1)\Big) \nonumber \\
=& (-1)^{k-1}\sum\limits_{j=-k}^{k-1}(-1)^{j}p(n-j(3j+1)/2) \nonumber \\
=& (-1)^{k-1}\sum\limits_{j=-k}^{k-1}(-1)^j\big(|A_{j}^{(1)}(n)|+|A_{j}^{(2)}(n)| \big) \nonumber \\
=& -(|A_{-k}^{(1)}(n)|+|A_{-k}^{(2)}(n)|)+(|A_{-k+1}^{(1)}(n)|+|A_{-k+1}^{(2)}(n)|)-(|A_{-k+2}^{(1)}(n)|\nonumber \\
&\quad +|A_{-k+2}^{(2)}(n)|)+\cdots -(|A_{k-2}^{(1)}(n)|+|A_{k-2}^{(2)}(n)|) +(|A_{k-1}^{(1)}(n)|+|A_{k-1}^{(2)}(n)|) \nonumber \\
=& |A_{k-1}^{(2)}(n)|-|A_{-k}^{(1)}(n)|.\label{Ak-proof}
\end{align}
This proves \eqref{Ak}. Now it suffices to prove that
\begin{align}\label{Ak-ineq}
|A_{k-1}^{(2)}(n)|\geq |A_{-k}^{(1)}(n)|.
\end{align}
For any partition $\lambda=(\lambda_1,\lambda_2,\cdots,\lambda_t)$, we denote by $\overline{\lambda}=(\lambda_1',\lambda_2',\cdots,\lambda_s')$ its conjugate partition. That is, $\lambda_i'$ is the number of parts in $\lambda$ that are $\geq i$. See \cite[pp. 7-8]{Andrews-book} for more discussion of conjugate partitions. Now for each partition $\lambda=(\lambda_1,\lambda_2,\cdots,\lambda_t)\in A_{-k}^{(1)}(n)$, we have
\begin{align}
\lambda_1+\lambda_2+\cdots+\lambda_t=n-\frac{k(3k-1)}{2}, \quad t-3k\geq \lambda_1.
\end{align}
Consider its conjugation $\overline{\lambda}=(\lambda_1',\lambda_2',\cdots,\lambda_s')$. We have $\lambda_1'\geq \lambda_2'\geq \cdots\ \geq \lambda_s'$, $\lambda_1'=t$, $s=\lambda_1$ and hence $\lambda_1'-s\geq 3k$. We define $$\psi(\lambda):=(\lambda_1'+2k-1,\lambda_2',\cdots,\lambda_s').$$
Clearly $\psi$ gives an injection from $A_{-k}^{(1)}(n)$ to $A_{k-1}^{(2)}(n)$.  Hence \eqref{Ak-ineq} holds. This completes the proof of Theorem \ref{thm-Euler-ineq}.
\end{proof}
\begin{rem}
We shall give an example illustrating the injection $\psi$. We choose $n=15$, $k=2$. We have
\begin{align*}
A_{-2}^{(1)}(15)=&\{(2,2,1,1,1,1,1,1),(2,1,1,1,1,1,1,1,1),(1,1,1,1,1,1,1,1,1,1)\}, \\
A_{1}^{(2)}(15)=&\{(13),(12,1),(11,2),(11,1,1),(10,3),(10,2,1),(10,1,1,1),(9,4),(9,3,1),\nonumber \\
&(9,2,2),(9,2,1,1),(9,1,1,1,),(8,5),(8,4,1),(8,3,2),(8,3,1,1),(8,2,2,1),\nonumber \\
&(7,6),(7,5,1),(7,4,2),(7,3,3)\}.
\end{align*}
The map $\psi$ works as:
\begin{align*}
\psi: (2,2,1,1,1,1,1,1) \quad (\mapsto \text{conjugation} ~~(8,2)) \quad \longmapsto (11,2), \\
\psi: (2,1,1,1,1,1,1,1,1) \quad (\mapsto \text{conjugation}~~ (9,1)) \quad \longmapsto (12,1), \\
\psi:(1,1,1,1,1,1,1,1,1,1) \quad (\mapsto \text{conjugation}~~ (10))  \quad \longmapsto (13).
\end{align*}
\end{rem}

\section{Proof of Theorem \ref{thm-main}}
This section is devoted to proving Theorem \ref{thm-main}. We will follow the method and use several results in the work of Mao \cite{Mao}.
\begin{proof}[Proof of Theorem \ref{thm-main}]
We denote
\begin{align}
D_k(q):=\frac{\sum\limits_{j=-k}^{k-1}(-1)^jjq^{Rj(j+1)/2-Sj}}{(q^S,q^{R-S},q^R;q^R)_{\infty}}-\sum\limits_{n=0}^{\infty}\Big(\frac{q^{nR+S}}{1-q^{nR+S}}-\frac{q^{nR+R-S}}{1-q^{nR+R-S}} \Big).
\end{align}
We have
\begin{align}
&-(q^S,q^{R-S},q^R;q^R)_\infty D_k(q)\nonumber \\
&= \left(\sum_{j=-\infty}^{-k-1} +\sum_{j=k}^\infty \right) (-1)^jjq^{\frac{R}{2}j^2+(\frac{R}{2}-S)j}  \quad \text{(by \eqref{diff-rec})}\nonumber \\
&=\sum_{j=k}^\infty \left((-1)^jjq^{\frac{R}{2}j^2+(\frac{R}{2}-S)j}+(-1)^{-j-1}(-j-1)q^{\frac{R}{2}(j+1)^2-(\frac{R}{2}-S)(j+1)} \right) \nonumber  \\
&=\sum_{j=0}^\infty \left((-1)^{j+k}(j+k)q^{\frac{R}{2}(j+k)^2+(\frac{R}{2}-S)(j+k)}+(-1)^{j+k}(j+k+1)q^{\frac{R}{2}(j+k)^2+(\frac{R}{2}+S)(j+k)+S} \right) \nonumber \\
&=(-1)^{k}q^{\frac{R}{2}k^2+(\frac{R}{2}-S)k}\sum_{j=0}^\infty \Big((-1)^j(j+k)q^{\frac{R}{2}j^2+Rjk+(\frac{R}{2}-S)j}\nonumber \\
&\quad \quad +(-1)^j(j+k+1)q^{\frac{R}{2}j^2+Rjk+(\frac{R}{2}+S)j+S(2k+1)}  \Big). \label{D-start}
\end{align}
Following \cite[Lemma 3.1]{Mao}, we set
\begin{align}
f_{R,S,k}:=(q^R,q^{Rk-S};q^R)_\infty \sum_{n=0}^\infty \frac{q^{Rn}}{(q^R,q^{Rk-S};q^R)_n}.
\end{align}
It was proved that \cite[Eqs.\ (3.3) and (3.4)]{Mao}
\begin{align}
\sum_{j=1}^\infty (-1)^{j+1}q^{Rj(j-1)/2+Rkj+Sj}&=1-f_{R,-S,k}, \label{Mao-id-1} \\
\sum_{j=1}^\infty (-1)^{j+1}q^{Rj(j-1)/2+Rkj-Sj}&=1-f_{R,S,k}. \label{Mao-id-2}
\end{align}
From \eqref{D-start} we can write
\begin{align}\label{D-split}
&(-1)^{k-1}(q^S,q^{R-S},q^R;q^R)_\infty D_k(q) \nonumber \\
=&q^{\frac{R}{2}k^2+(\frac{R}{2}-S)k}\Big((k-1)I_1(q)+kI_2(q)+I_3(q)+I_4(q)\Big),
\end{align}
where
\begin{align}
I_1(q)&:=\sum_{j=0}^\infty (-1)^jq^{\frac{R}{2}(j^2+j)+(kR-S)j}, \label{I1-def} \\
I_2(q)&:=\sum_{j=0}^\infty (-1)^jq^{\frac{R}{2}j^2+Rjk+(\frac{R}{2}+S)j+S(2k+1)}, \label{I2-def} \\
I_3(q)&:=\sum_{j=0}^\infty (-1)^j (j+1)q^{\frac{R}{2}j^2+Rjk+(\frac{R}{2}-S)j}, \label{I3-def}\\
I_4(q)&:=\sum_{j=0}^\infty (-1)^j(j+1)q^{\frac{R}{2}j^2+Rjk+(\frac{R}{2}+S)j+S(2k+1)}. \label{I4-def}
\end{align}
We shall treat $I_1(q), I_2(q), I_3(q)$ and  $I_4(q)$ one by one.

First, replacing $j$ by $j-1$ in \eqref{I1-def}, we obtain
\begin{align}\label{I1-exp}
I_1(q)=q^{S-kR}\sum_{j=1}^\infty (-1)^{j+1}q^{\frac{R}{2}j(j-1)+(kR-S)j}  =q^{S-kR}\left(1-f_{R,S,k} \right).
\end{align}
Similarly, we have
\begin{align}
I_2(q)&=\sum_{j=1}^\infty (-1)^{j+1}q^{\frac{R}{2}j(j-1)+(kR+S)j+(2S-R)k}=q^{(2S-R)k}(1-f_{R,-S,k}). \label{I2-exp}
\end{align}
%\begin{align}
%\frac{I_1(q)}{(q^S,q^{R-S},q^R;q^R)_\infty}=q^{S-kR}\frac{(1+q^{(2k-1)S})(1-f_{R,-S,k})+(f_{R,-S,k}-f_{R,S,k})}{(q^S,q^{R-S},q^R;q^R)_\infty}. \label{I1-split}
%\end{align}
%By \cite[Lemma 3.1]{Mao} we know that
%\begin{align}
%\frac{f_{R,-S,k}-f_{R,S,k}}{(q^S,q^{R-S},q^R;q^R)_\infty}\succeq 0. \label{I1-a}
%\end{align}
From \eqref{Mao-id-2}  we have
\begin{align}
&\frac{1-f_{R,S,k}}{(q^S,q^{R-S},q^R;q^R)_\infty}\nonumber \\
=&\frac{\sum_{j=1}^\infty (-1)^{j+1}q^{Rj(j-1)/2+Rkj-Sj}}{(q^R,q^{R-S},q^S;q^R)_\infty} \quad \text{(split the sum according to $j$ even and odd)} \nonumber \\
=&\frac{1}{(q^R,q^{S};q^R)_\infty}\sum_{j=0}^\infty
q^{Rj(2j+1)+Rk(2j+1)-S(2j+1)}\times \frac{1-q^{R(2j+1)+Rk-S}}{(q^{R-S};q^R)_\infty} \succeq 0. \label{I1-mid}
\end{align}
Similarly, (see also \cite[p.\ 21]{Mao})
\begin{align}
&\frac{1-f_{R,-S,k}}{(q^S,q^{R-S},q^R;q^R)_\infty}\nonumber \\
=&\frac{\sum_{j=1}^\infty (-1)^{j+1}q^{Rj(j-1)/2+Rkj+Sj}}{(q^R,q^{R-S},q^S;q^R)_\infty} \nonumber \\
=&\frac{1}{(q^R,q^{R-S};q^R)_\infty}\sum_{j=0}^\infty
q^{Rj(2j+1)+Rk(2j+1)+S(2j+1)}\times \frac{1-q^{R(2j+1)+Rk+S}}{(q^S;q^R)_\infty} \succeq 0. \label{I2-mid}
\end{align}
From \eqref{I1-exp}--\eqref{I2-mid} we deduce that
\begin{align}
\frac{I_1(q)}{(q^S,q^{R-S},q^R;q^R)_\infty} &\succeq 0, \label{I1-positive} \\
\frac{I_2(q)}{(q^S,q^{R-S},q^R;q^R)_\infty} &\succeq 0. \label{I2-positive}
\end{align}

Now we turn to $I_3(q)$. From \cite[p.\ 22]{Mao} we find
\begin{align}\label{Mao-key}
\sum_{j=0}^\infty (-1)^j(j+1)a^jq^{j(j+1)/2}=(a,q;q)_\infty \sum_{n=0}^\infty \sum_{m=0}^\infty \frac{q^{2n+m}}{(a,q;q)_n(1-aq^{n+m})}.
\end{align}
Replacing $q$ by $q^R$ and setting $a=q^{kR-S}$, we obtain
\begin{align}
I_3(q)&=\sum_{j=0}^\infty (-1)^{j}(j+1)q^{\frac{R}{2}j(j+1)+(kR-S)j}\nonumber \\
&=(q^{kR-S},q^R;q^R)_\infty \sum_{n=0}^\infty \sum_{m=0}^\infty \frac{q^{R(2n+m)}}{(q^{kR-S},q^R;q^R)_n(1-q^{kR-S+R(n+m)})}. \label{I3-exp}
\end{align}
It follows that
\begin{align}
&\frac{I_3(q)}{(q^S,q^{R-S},q^R;q^R)_\infty}\nonumber \\
=&\frac{1}{(q^S;q^R)_\infty (q^{R-S};q^R)_{k-1}}\sum_{n=0}^\infty \sum_{m=0}^\infty \frac{q^{R(2n+m)}}{(q^{kR-S},q^R;q^R)_n(1-q^{kR-S+R(n+m)})} \succeq 0. \label{I3-positive}
\end{align}
Similarly, replacing $q$ by $q^R$ and setting $a=q^{kR+S}$ in \eqref{Mao-key}, we obtain
\begin{align}
I_4(q)&=\sum_{j=0}^\infty (-1)^j(j+1)q^{\frac{R}{2}j(j+1)+(kR+S)j+S(2k+1)} \nonumber \\
&=q^{S(2k+1)}(q^{kR+S},q^R;q^R)_\infty \sum_{n=0}^\infty \sum_{m=0}^\infty \frac{q^{R(2n+m)}}{(q^{kR+S},q^R;q^R)_n(1-q^{kR+S+R(n+m)})}. \label{I4-exp}
\end{align}
It follows that
\begin{align}
&\frac{I_4(q)}{(q^S,q^{R-S},q^R;q^R)_\infty}\nonumber \\
=&\frac{q^{S(2k+1)}}{(q^{R-S};q^R)_\infty (q^{S};q^R)_{k}}\sum_{n=0}^\infty \sum_{m=0}^\infty \frac{q^{R(2n+m)}}{(q^{kR+S},q^R;q^R)_n(1-q^{kR+S+R(n+m)})} \succeq 0. \label{I4-positive}
\end{align}
Combining \eqref{I1-positive}, \eqref{I2-positive}, \eqref{I3-positive} and \eqref{I4-positive}, from \eqref{D-split} we conclude that
\begin{align*}
(-1)^{k-1} D_k(q)\succeq 0.
\end{align*}
This completes the proof.
\end{proof}

We end this paper with three problems for investigation in the future. From \eqref{Andrews-Merca-id} and \eqref{Ak} we conclude that
\begin{align}\label{Mk-Ak}
M_k(n)=|A_{k-1}^{(2)}(n)|-|A_{-k}^{(1)}(n)|.
\end{align}
\noindent \textbf{Problem 1.} Can one prove \eqref{Mk-Ak} directly and perhaps in a combinatorial way?\\
\noindent \textbf{Problem 2.} The following identity can be viewed as a generalization of Euler's pentagonal number theorem:
\begin{align}
    (q^S,q^{R-S},q^R;q^R)_\infty=\sum_{j=-\infty}^\infty (-1)^jq^{Rj(j+1)/2-Sj}.
\end{align}
Can one find a mapping generalizing the involution $\phi$ of Bressoud and Zeilberger to prove this identity combinatorially?
Moreover, can one prove Conjecture \ref{AMconj} by using such mapping? \\
\noindent \textbf{Problem 3.} Is it possible to find an identity in analogy with \eqref{Wang-Yee-id} which proves Theorem \ref{thm-main}?

While there might be different approaches in solving the above problems, we have some expectations for the solutions. We hope that someone can find a solution of Problem 2 by establishing a result similar to Theorem \ref{thm-Euler-ineq}. We also expect that the identity for Problem 3 can provide combinatorial interpretation of the coefficients in the series in Theorem \ref{thm-main}, which can naturally explain the nonnegativity.

\subsection*{Acknowledgements}
This work was supported by the National Natural Science Foundation of China (11801424) and a start-up research grant of the Wuhan University.

\end{document}